\documentclass{amsart}

\usepackage{macros}
\standardmargins\standardpackages\standardrenews
\standardlabeling\standardmakes
\colorcommentstrue
\usepackage[utf8]{inputenc}
 \usepackage{amssymb}
\let\comjohannes\undefined

\draftfalse

\newcommand{\Bad}{\mathbf{Bad}}
\newcommand{\eps}{\varepsilon}
\newcommand{\normZ}[1]{\left\|#1\right\|_{\mathbb Z}}

\title{Well and badly approximable sets, and rapid winning}
\authormumtaz
\authordavid

\begin{document}
\maketitle

\begin{abstract}
The set of $\tau$-approximable numbers, $\mathcal W(\tau)$, has genuinely fractional Hausdorff dimension, whereas the set of inhomogeneously badly approximable
numbers, $\Bad^\gamma$,  has full Hausdorff dimension. We determine the Hausdorff dimension of their intersection by introducing the $\Psi$-rapid game, a scale-sensitive refinement of the rapid game of Hatefi and Simmons (preprint 2024). For every approximation function $\psi$ and $\gamma\in\R\setminus\Z$,  we prove that $\mathcal W(\psi)\cap\Bad^\gamma$ is strong $\Psi$-rapid winning for a natural gauge $\Psi$ determined by $\psi$. Unlike Schmidt-type games, whose winning property always implies full Hausdorff dimension, the $\Psi$-rapid game is calibrated to a prescribed Diophantine scale, so that the resulting dimension bound depends explicitly on the decay of $\Psi$. In particular, for $\psi(q)=q^{-\tau}, \tau\ge1,$ we obtain the exact Jarn\'ik--Besicovitch dimension, that is, 
$$
\HD\bigl(\mathcal W(\tau)\cap\Bad^\gamma\bigr)=\frac{2}{\tau+1}.$$
\end{abstract}

\section{Introduction and statement of results}

Let $\psi:\mathbb{N}\to(0,\infty)$ be a function such that $\psi(n)\to 0$ as $n\to\infty$. The
well-known set of $\psi$-approximable numbers is defined as
\[
\WW(\psi) = \left\{x\in\mathbb{R}: \normZ{qx}<\psi(q)
\text{ for infinitely many } q\in\mathbb{N}\right\}.
\]
Here and throughout, $\normZ{\cdot}$ denotes the distance to the nearest integer. The beginning point of the metrical theory of this set is the following classical result.

\begin{theorem}[Khintchine, 1924]
\label{thm:khintchine}
Let $\psi:\mathbb N\to(0,\infty)$ be a non-increasing function. Then
\[
\lambda(\WW(\psi)\cap [0,1))
=
\begin{cases}
0 & \text{if } \sum_{q=1}^\infty \psi(q)<\infty,\\
1 & \text{if } \sum_{q=1}^\infty \psi(q)=\infty,
\end{cases}
\]
where $\lambda$ denotes the Lebesgue measure.
\end{theorem}
One of the major recent developments in this direction is the resolution
of the Duffin--Schaeffer conjecture (1941) by Koukoulopoulos and Maynard
\cite{KM20}. The conjecture asked for the same
zero--one law as Khintchine's theorem without the monotonicity assumption
on \(\psi\), provided that the natural coprimality condition on the
rational approximants is imposed. Note that although Khintchine's theorem is a remarkable result, it does not distinguish between different null sets $\WW(\psi)$ arising for fast decreasing function $\psi$. The
classical refinement is provided by two independent result by Jarn\'ik (1928) and Besicovitch (1934), now commonly known as Jarn\'ik--Besicovitch theorem. Whenever $\psi_\tau(q):=q^{-\tau}$, we write $ \WW(\tau):=\WW(\psi_\tau).$
\begin{theorem}[Jarn\'ik--Besicovitch, 1928--34]
\label{thm:jb}
For any $\tau\geq 1$, we have
\[
\HD \WW(\tau) = \frac{2}{\tau+1}.
\]
\end{theorem}
Thus, Theorem~\ref{thm:jb} confirms the intuition that the faster
\(\psi_\tau\) decays, the smaller the corresponding approximation set
\(\WW(\tau)\) becomes in terms of Hausdorff dimension. A natural complementary family consists of the badly approximable numbers
\[\Bad = \left\{x\in\mathbb{R}: \inf_{q\ge1} q\normZ{qx}>0\right\},\]
and their inhomogeneous analogues
\[\Bad^\gamma = \left\{x\in\mathbb R: \inf_{q\ge1} q\normZ{qx-\gamma}>0\right\}.\]
Both of these sets have full Hausdorff dimension; moreover they satisfy much stronger largeness properties in the sense of Schmidt games and their refinements \cite{BadziahinHarrap, BNY,  BFKRW, DattaShao, EinsiedlerTseng, FSU4,HatefiSimmons2024, HLWY, McMullen2010, Schmidt1}.

\smallskip

One of the most important features of Schmidt-winning sets is their
stability under countable intersections. In particular,
$
\Bad\cap\Bad^\gamma$ 
is again winning. Consequently,
$
\Bad^\gamma\setminus\Bad$ 
cannot be winning in the classical Schmidt sense, since it is disjoint
from another winning set. In \cite{HatefiSimmons2024}, Hatefi and Simmons introduced a variant of the Schmidt game, called the {\em rapid game}, which does
not have the intersection property but for which ``winning'' still implies full dimension. Classical Schmidt-type winning properties are typically associated with
sets of full Hausdorff dimension. This raises the following natural
question.
\[
\text{\em Can a genuinely fractional-dimensional Diophantine set carry a natural winning property?}
\]
We answer this question affirmatively by introducing the
\(\Psi\)-rapid game, a scale-sensitive refinement of the rapid game of
Hatefi and Simmons adapted to a prescribed approximation function
\(\psi\). A consequence of this new game is the following result.

\begin{theorem}\label{cor:jb}For all $\gamma\in(\R/\Z)\setminus\{0\}$ and $\tau\geq 1$, we have 
\[
\HD\bigl(\WW(\tau)\cap\Bad^\gamma\bigr) = \frac{2}{\tau+1}.
\] 
\end{theorem}

The set $\WW(\psi)\cap\Bad^\gamma$ may be viewed as a natural
$\psi$-analogue of the set $\Bad^\gamma\setminus\Bad$ studied in
\cite{HatefiSimmons2024}. 
It consists of points
$x\in[0,1)$ whose orbit $qx \pmod 1$ avoids the shrinking target
$
(\gamma-\tfrac{c(x)}{q},\,\gamma+\tfrac{c(x)}{q})$ 
for some $c(x)>0$, while simultaneously visiting the neighbourhood $
 (-\tfrac{\psi(q)}{q},\,\tfrac{\psi(q)}{q})$
of the origin infinitely~often. 
\begin{remark}
If $\gamma\equiv0\pmod 1$, then $\Bad^\gamma=\Bad$. Clearly, for $\tau>1$,
the intersection $\WW(\tau)\cap\Bad$ is empty. Thus the most interesting
case is $\gamma\not\equiv0\pmod 1$.
\end{remark}

\subsection{The Strong $\Psi$-Rapid Game}
\label{sec:game}

Let $X$ be a complete metric space. For $x\in X$ and $\rho>0$, write
$$B(x,\rho)=\{y\in X:d(x,y)\leq \rho\}.$$ Let $S\subseteq X$ be Alice's
target set. Fix a gauge $\Psi:(0,\rho_0]\to(0,1]$ and parameters
$0<\alpha,\beta<1$. Following McMullen \cite{McMullen2010}, we introduce
the strong $\Psi$-rapid game.

\begin{definition}[Strong $\Psi$-rapid game]
\label{def:game}
A play of the strong $\Psi$-rapid game consists of nested closed balls
\[
B_0\supseteq A_0\supseteq B_1\supseteq A_1\supseteq B_2\supseteq\cdots.
\]
At the $0$th stage, Bob begins by choosing $B_0=B(x_0,\rho_0)$ together
with a parameter $\alpha_0\in(0,\alpha)$. At the $n$th stage, after Bob
has chosen the ball $B_n=B(x_n,\rho_n)$, Bob also announces a number
$\alpha_n\in(0,\alpha)$. Then Alice chooses a ball
$A_n=B(y_n,\rho'_n)\subseteq B_n$ with $\rho'_n\geq \alpha_n\rho_n$. At
the $(n+1)$st stage, Bob chooses
\[
B_{n+1}=B(x_{n+1},\rho_{n+1})\subseteq A_n
\]
with $\rho_{n+1}\geq \beta\rho'_n$. Bob loses by default if the radii do not tend to zero, that is if
$\inf_n\alpha_n>0$. If the radii tend to zero, the nested balls have a
unique intersection point
\[
\{x_\infty\}=\bigcap_{n=0}^\infty B_n.
\]
Bob is required to play $\Psi$-rapidly:
\[
\liminf_{n\to\infty}\frac{\alpha_n}{\Psi(\rho_n)}=0.
\]
Alice wins if either
\[
\liminf_{n\to\infty}\frac{\alpha_n}{\Psi(\rho_n)}>0,
\]
so that Bob loses by default, or else the unique intersection point
$x_\infty$ is in the target set $S$.
\end{definition}

\begin{definition}[Strong $\Psi$-rapid winning]
\label{def:winning}
A set $S\subseteq X$ is called strong $(\Psi,\alpha,\beta)$-rapid winning
on $X$ if Alice has a winning strategy for the strong
$(\Psi,\alpha,\beta)$-rapid game with target $S$. It is called strong
$\Psi$-rapid winning if this holds for every sufficiently small
$\beta>0$.
\end{definition}

\begin{remark}
If $\Psi\equiv1$, then the default condition becomes
$\liminf_{n\to\infty}\alpha_n=0$. Thus the strong $\Psi$-rapid game reduces
to the strong rapid game of \cite{HatefiSimmons2024}. This is the case
corresponding to the usual Dirichlet exponent in dimension one.
\end{remark}

\begin{remark} The strong game gives Bob greater flexibility than the equality-radius
game. Therefore
\[
\text{strong $\Psi$-rapid winning} \Longrightarrow
\text{ordinary $\Psi$-rapid winning}.
\]
The strong formulation is convenient in the grid proof because the passage
between game balls and diagonal-flow times is stable up to bounded changes
of scale. All lattice and grid estimates below are uniform under such
bounded perturbations. This is the same role played by the strong game in
the rapid-game proof.
\end{remark}

In the strong \(\Psi\)-rapid game, Bob's parameter \(\alpha_n\) measures
the scale at which Alice must respond. At scale \(\rho_n\), a successful move produces a denominator
\(q\asymp \rho_n^{-1/2}\), and the resulting approximation error is of
size \(\lessless q^{-1}\alpha_n\).
Thus the condition \(\alpha_n\lessless q\psi(q)\) is sufficient to force
\[
\normZ{qx}<\psi(q),
\]
which naturally leads to the scale
\[
\Psi_\psi(\rho)\asymp \rho^{-1/2}\psi(\rho^{-1/2}).
\]

To make this precise, fix constants \(0<a<b<\infty\). For \(Q\ge1\),
define
\[
\Phi_\psi(Q)=\inf_{aQ\le q\le bQ} q\psi(q).
\]
Choose a sufficiently small constant \(c_*>0\), and set
\[
\Psi_\psi(\rho)=c_*\Phi_\psi(\rho^{-1/2}).
\]
In particular, if \(\psi(q)=q^{-\tau}\), then
\[
\Psi_\psi(\rho)\asymp \rho^{(\tau-1)/2}.
\]

We may now state the main result 
that shows that
$\WW(\psi)\cap\Bad^\gamma$ 
is strong \(\Psi\)-rapid winning.

\begin{theorem}\label{thm:main}
Let $\gamma\in\R/\Z$ and let $\psi:\N\to(0,\infty)$ be an
approximation function. Let $\Psi_\psi$ be defined as above. Then
$\WW(\psi)\cap\Bad^\gamma$ is strong $\Psi_\psi$-rapid winning on every
nonempty interval.
\end{theorem}

We will see that Theorem~\ref{cor:jb} follows swiftly from Theorems ~\ref{thm:main} and \ref{thm:jb}. 



\medskip

The proof follows the same broad strategy as the rapid-game argument of
Hatefi--Simmons \cite{HatefiSimmons2024}, but with one important
modification: the auxiliary move is now only a one-way cusp-entry move, while the
return to a bounded region is delegated to the default strategy. This
separation is important to eliminate the square-root loss and is what leads to the sharp
Jarn\'ik--Besicovitch scaling.

The remainder of the paper is organised as follows.
In Section~2, we prove the dimension consequence of the $\Psi$-rapid game. 
Section~3 introduces the homogeneous and inhomogeneous lattice machinery.
Section~4 establishes the winning property of
\(\WW(\psi)\cap\Bad^\gamma\).
Section~5 is reserved for the proof of  Theorem \ref{cor:jb} and we end the paper with some related open problems in Section \ref{conclusion}.
\section{The Dimension Proposition}

In this section, we prove the Hausdorff dimension consequence of
the $\Psi$-rapid winning.

\begin{proposition}
\label{prop:dimension}
Let $I_0\subset\mathbb{R}$ be a non-empty closed interval. Let
\[
\Psi_\omega(\rho)=\rho^\omega, \qquad \omega\geq0.
\]
Fix $\alpha,\beta > 0$, and suppose $S\subseteq I_0$ is strong $(\Psi_\omega,\alpha,\beta)$-rapid winning.
Fix an ordinary contraction ratio $\alpha_0\in(0,\alpha)$, and let $m_\beta$ be
an integer such that after each Alice move Bob has at least $m_\beta$
pairwise disjoint legal replies of radius at least $\beta$ times Alice's
radius. Set
\[
\Delta_\beta = \frac{\log m_\beta}{\log(1/(\alpha_0\beta))}.
\]
Then
\[
\HD S \geq \frac{\Delta_\beta}{1+\omega}.
\]
In particular, for any fixed $\alpha_0\in(0,\alpha)$, taking
\[
m_\beta=\left\lfloor\frac{1}{4\alpha_0\beta}\right\rfloor,
\]
one has $\Delta_\beta\to1$ as $\beta\to0$. Hence if $S$ is strong
$(\Psi_\omega,\alpha,\beta)$-rapid winning for some $\alpha > 0$ and for every sufficiently small
$\beta>0$, then
\[
\HD S \geq \frac{1}{1+\omega}.
\]
\end{proposition}
\begin{proof}
Fix an initial interval $B_0=I_0=B(x_0,\rho_0)$. Let $\sigma$ be Alice's
winning strategy for the $(\Psi_\omega,\alpha,\beta)$-rapid game. We construct a
Cantor set $K\subset S$ by allowing Bob to play a tree of legal moves
against the fixed strategy $\sigma$. The construction has two types of
rounds.

\medskip

\noindent
\textbf{Ordinary rounds.} At an ordinary round, Bob chooses a fixed number
$0<\alpha_0<\alpha$. Alice then responds, according to $\sigma$, with an
interval $A_n\subset B_n$ with $r(A_n)=\alpha_0\rho_n$. Inside $A_n$, Bob
has at least $m_\beta$ pairwise disjoint legal choices for $B_{n+1}$, each
of radius $\rho_{n+1}=\beta\alpha_0\rho_n$. The exact constant in $m_\beta$
is not important; all that matters is that $m_\beta\asymp \beta^{-1}$.

For definiteness, one may take
\[
m_\beta=\left\lfloor\frac{1}{4\alpha_0\beta}\right\rfloor
\]
when $0<\beta<1/(10\alpha_0)$. Note that any interval of length $L$ contains at least $\lfloor L/(4\alpha_0\beta r(A_n))\rfloor$ disjoint subintervals of length $\beta r(A_n)$ separated by gaps of at least their own length.

\medskip

\noindent
\textbf{Microscopic rounds.} At the $j$th microscopic round, Bob chooses
$\alpha_n=\eta_j \rho_n^\omega$, where $(\eta_j)$ is a sequence tending to zero sufficiently slowly, as described below. Since
$\rho_n\to0$, after discarding finitely many initial rounds we may assume
that
\[
0<\eta_j\rho_n^\omega<\alpha,
\]
so this is a legal Bob move. Alice responds with
\[
A_n\subset B_n, \qquad r(A_n)=\eta_j\rho_n^{1+\omega}.
\]
Bob again chooses one of at least $m_\beta$ disjoint legal replies inside
$A_n$, each of radius
\[
\rho_{n+1} = \beta\eta_j\rho_n^{1+\omega}.
\]

The key point is that along every infinite branch of this tree,
\[
\frac{\alpha_n}{\Psi_\omega(\rho_n)}
=
\frac{\eta_j\rho_n^\omega}{\rho_n^\omega}
=
\eta_j
\longrightarrow0
\]
at the microscopic rounds. Hence every branch satisfies the default
condition
\[
\liminf_{n\to\infty}\frac{\alpha_n}{\Psi_\omega(\rho_n)}=0.
\]
Therefore Alice's winning strategy forces the limit point of every branch
to lie in $S$. Thus the Cantor set $K$ constructed from all branches satisfies
$K\subset S$.

\smallskip

It remains to estimate the dimension of $K$.

\smallskip

Let $N_n$ denote the number of level $n$ intervals in the construction.
Since every interval has exactly $m_\beta$ children, we have
$N_n=m_\beta^n$. Let $\rho_n$ denote the common radius of all level $n$
intervals, and put $R_n:=-\log\rho_n$. We will prove that the microscopic
rounds can be arranged so that
\[
\liminf_{n\to\infty}\frac{\log N_n}{R_n}
\geq
\frac{\Delta_\beta}{1+\omega}.
\]
This will imply $\HD K\geq \Delta_\beta/(1+\omega)$. Indeed, once the above
inequality is known, let $s<\Delta_\beta/(1+\omega)$. Then, for all
sufficiently large $n$, $N_n^{-1}=m_\beta^{-n}\leq \rho_n^s$. Put the
uniform probability measure $\mu$ on $K$, assigning mass $N_n^{-1}$ to
each level $n$ interval.

Let $J\subset\mathbb R$ be any interval with sufficiently small radius.
Choose $n$ such that $\rho_n\leq r(J)<\rho_{n-1}$. Since the level $n-1$
construction intervals are pairwise disjoint and all have radius
$\rho_{n-1}$, the interval $J$ intersects at most a bounded number, say
$C$, of level $n-1$ intervals. Hence
\[
\mu(J)
\leq C N_{n-1}^{-1}
= C m_\beta N_n^{-1}
\leq C m_\beta \rho_n^s.
\]
Since $\rho_n\leq r(J)$, this gives $\mu(J)\leq C m_\beta r(J)^s$. Thus
$\mu$ is an $s$-Frostman measure, and so $\HD K\geq s$. Letting
$s\to \Delta_\beta/(1+\omega)$ gives the desired lower bound.

We now verify the logarithmic estimate. Let $h:=\log(1/(\alpha_0\beta))$.
During an ordinary round, we have
\[
R_{n+1}=R_n+h,
\qquad
\log N_{n+1}=\log N_n+\log m_\beta.
\]
Thus repeated ordinary rounds drive the ratio $\log N_n/R_n$ towards
$\Delta_\beta=\log m_\beta/h$.

We choose the construction in blocks. Suppose the $(j-1)$th microscopic
round has been completed. We then play sufficiently many ordinary rounds so
that, just before the $j$th microscopic round, at some level $n_j$, we
have
\[
\frac{\log N_{n_j}}{R_{n_j}}
\geq
\Delta_\beta-\varepsilon_j,
\]
where $\varepsilon_j\to0$. Now perform the $j$th microscopic round, with
$\alpha_{n_j}=\eta_j\rho_{n_j}^\omega$. Then
\[
\rho_{n_j+1}
=
\beta\eta_j\rho_{n_j}^{1+\omega}.
\]
Taking logarithms gives
\[
R_{n_j+1}
=
(1+\omega)R_{n_j}+\log(1/(\beta\eta_j)),
\qquad
\log N_{n_j+1}
=
\log N_{n_j}+\log m_\beta.
\]
We choose $\eta_j\to0$ so slowly that $\log(1/\eta_j)=o(R_{n_j})$. For
instance, after $R_{n_j}$ is known, one may take
$\eta_j=\exp(-\sqrt{R_{n_j}})$. Then
\[
\frac{\log(1/(\beta\eta_j))}{R_{n_j}}\longrightarrow0.
\]
Therefore
\[
\frac{\log N_{n_j+1}}{R_{n_j+1}}
=
\frac{\log N_{n_j}+\log m_\beta}
{(1+\omega)R_{n_j}+\log(1/(\beta\eta_j))}
\geq
\frac{\Delta_\beta-\varepsilon_j+o(1)}
{1+\omega+o(1)}.
\]
Hence, after increasing $j$ if necessary,
\[
\frac{\log N_{n_j+1}}{R_{n_j+1}}
\geq
\frac{\Delta_\beta}{1+\omega}-2\varepsilon_j.
\]

Between microscopic rounds we again play ordinary rounds. If the current
ratio is below $\Delta_\beta$, adding ordinary rounds replaces the ratio
by a weighted average with $\Delta_\beta$, and therefore increases it.
Consequently the smallest ratios occur immediately after microscopic
rounds. It follows that
\[
\liminf_{n\to\infty}\frac{\log N_n}{R_n}
\geq
\frac{\Delta_\beta}{1+\omega}.
\]
As shown above, this implies $\HD K\geq \Delta_\beta/(1+\omega)$. Since
$K\subset S$, we obtain $\HD S\geq \Delta_\beta/(1+\omega)$.

Finally, if $S$ is $(\Psi_\omega,\alpha,\beta)$-rapid winning for all
sufficiently small $\beta$, choose any fixed $\alpha_0\in(0,\alpha)$ and take
\[
m_\beta=\left\lfloor\frac{1}{4\alpha_0\beta}\right\rfloor.
\]
Then
\[
\Delta_\beta
=
\frac{\log \left\lfloor\frac{1}{4\alpha_0\beta}\right\rfloor}{\log(1/(\alpha_0\beta))}
\longrightarrow1
\qquad(\beta\to0).
\]
Thus $\HD S\geq 1/(1+\omega)$. This proves the proposition.
\end{proof}
The key observation is that the limit $\Delta_\beta\to1$ as $\beta\to0$ is independent of the particular choice of $\alpha_0$. Indeed, if $
m_\beta=\left\lfloor\frac{1}{c_\beta}\right\rfloor$ 
for any $c_\beta$ satisfying $c_\beta\asymp \alpha_0\beta$, then
\[
\Delta_\beta
=
\frac{\log m_\beta}{\log(1/(\alpha_0\beta))}
=
\frac{-\log(\alpha_0\beta)+O(1)}{-\log(\alpha_0\beta)}
\longrightarrow1.
\]
Thus we may choose $\alpha_0$ arbitrarily close to $\alpha$ to maximise the number of legal replies. If $\alpha$ is very small, then $\alpha_0$ must also be very small. This only affects the constant $m_\beta$ by a multiplicative factor, which does not change the limit $\Delta_\beta\to1$.
\begin{corollary}\label{cor:gauge-comparable}
Suppose $\Psi$ satisfies
\[
c_1\rho^\omega\leq \Psi(\rho)\leq c_2\rho^\omega
\]
for all sufficiently small $\rho$, with constants $c_1,c_2>0$. If $S$ is
strong $(\Psi,\alpha,\beta)$-rapid winning for every sufficiently small
$\beta>0$, then $\HD S\geq 1/(1+\omega)$.
\end{corollary}

\begin{proof}
The proof is unchanged from Proposition~\ref{prop:dimension}. At the
microscopic rounds choose $\alpha_n=\eta_j\rho_n^\omega$. Then
\[
\frac{\alpha_n}{\Psi(\rho_n)}
\leq
\frac{\eta_j}{c_1}\to0.
\]
The logarithmic radius calculation is unchanged up to bounded additive
errors.
\end{proof}

\section{Games on Grids}
\label{sec:grids}

In this section we reformulate the strong $\Psi$-rapid game in terms of
trajectories on the space of unimodular lattices and grids.

\subsection{Lattice and Grid Notation}

Let
\[
G=\SL_2(\R), \qquad \Gamma=\SL_2(\Z), \qquad X=G/\Gamma.
\]
Thus $X$ is (isomorphic to) the space of unimodular lattices in $\R^2$. Let
\[
Y=\{\Lambda+r:\Lambda\in X,\ r\in\R^2/\Lambda\}
\]
be the corresponding space of unimodular grids, and let
\[
\pi:Y\to X, \qquad \pi(\Lambda+r)=\Lambda,
\]
denote the natural projection. For $t\in\R$, define
\[
g_t=
\begin{pmatrix}
e^t&0\\
0&e^{-t}
\end{pmatrix},
\]
and for $x\in\R$, define
\[
u_x=
\begin{pmatrix}
1&-x\\
0&1
\end{pmatrix}.
\]
For $x\in\R$, define the homogeneous lattice $$\Lambda_x=u_x\Z^2=\{(p-qx,q): p, q\in\Z\}\in X.$$ For a
fixed shift $\gamma\in\R/\Z$, define the associated inhomogeneous grid
\[
\Lambda_x^\gamma = u_x(\Z^2+(\gamma,0))=\{(p+\gamma-qx,q): p, q\in\Z\}.
\]
For $\eps>0$, define
\[
\FF_\eps
=
\{
\Lambda+r\in Y:
(\Lambda+r)\cap B_\infty(0,\eps)=\emptyset
\}.
\]
Thus, $\FF_\eps$ consists of those grids whose distance from the origin is at
least $\eps$ in the sup norm.

For $\Lambda\in X$, define the first successive minimum
\[
\lambda_1(\Lambda)
=
\inf\{\|v\|_\infty:v\in\Lambda\setminus\{0\}\}.
\]
For $\delta>0$, define
\[
\mathcal K_\delta
=
\{\Lambda\in X:\lambda_1(\Lambda)\ge\delta\}.
\]

For $\Lambda+r\in Y$, define
\[
\Delta(\Lambda+r)
=
\inf_{v\in\Lambda+r}\|v\|_\infty.
\]
For $\delta>0$, define
\[
\mathcal F_\delta
=
\{\Lambda+r\in Y:\Delta(\Lambda+r)\ge\delta\}.
\]

Thus $\mathcal K_\delta$ is the compact set of lattices avoiding short
nonzero vectors and $\mathcal F_\delta$ is the corresponding compact set of
grids avoiding the origin. The notation
$\pi^{-1}(\mathcal K_\delta)$ is used for the set of grids whose
homogeneous part belongs to $\mathcal K_\delta$. Thus
\[
\Lambda+r\in \pi^{-1}(\mathcal K_\delta)
\quad\Longleftrightarrow\quad
\Lambda\in \mathcal K_\delta.
\]

For later use, if $a\in\Lambda\setminus\{0\}$, define the wedge-separation
quantity
\[
\Delta_a(\Lambda+r)
=
\inf_{b\in\Lambda+r}|a\wedge b|,
\]
where $a\wedge b=a_1b_2-a_2b_1$ is the wedge product.

The following result is the dynamical reformulation of $\Bad$ and
$\Bad^\gamma$.

\begin{lemma}[Dani correspondence]
\label{lem:dani}
For $x\in\R$, the following are equivalent:
\begin{enumerate}
\item $x\in\Bad$;
\item the orbit $\{g_t\Lambda_x:t\ge0\}$ is bounded in $X$.
\end{enumerate}
Moreover, $x\in\Bad^\gamma$ if and only if there exists $\eps>0$ such that
$g_t\Lambda_x^\gamma\in \FF_\eps$ for all sufficiently large $t\ge0$.
\end{lemma}

\begin{proof}
The homogeneous statement is the classical Dani correspondence \cite{Dani4}.
For the inhomogeneous statement, a point of $g_t\Lambda_x^\gamma$ has the
form
\[
(e^t(p+\gamma-qx),e^{-t}q), \qquad (p,q)\in\Z^2.
\]
Minimizing the sup norm in $t$ shows that the orbit avoids a neighbourhood
of the origin if and only if $q\normZ{qx-\gamma}$ is bounded away from
zero.
\end{proof}

We also require a compactness statement.

\begin{lemma}[Lattice minimum]
\label{lem:lattice-minimum}
For every $\delta>0$ there exist constants $0<c_0<C_0<\infty$ such that
every $\Lambda\in\mathcal K_\delta$ contains a primitive vector
$a=(a_1,a_2)$ satisfying
\[
c_0\le |a_2|\le C_0, \qquad
\left|\frac{a_1}{a_2}\right| \le \frac12.
\]
\end{lemma}

\begin{proof}
The proof is identical to the construction in
\cite[Section~4.2.3]{HatefiSimmons2024}, where the required primitive
vector is selected uniformly over the compact family
$\mathcal K_\delta$.
\end{proof}
\ignore{
By Mahler's compactness criterion, 
\[
\mathcal K_\delta = \{\Lambda\in X:\lambda_1(\Lambda)\ge\delta\}
\] is compact. For each $\Lambda\in\mathcal K_\delta$, the Gauss reduction algorithm \comdavid{I still don't know what you mean by ``Gauss reduction algorithm''.} 
produces a primitive vector $\aa=(a_1,a_2)\in\Lambda$ \comdavid{I changed the notation so that vectors are bold} such that
\[
\left|\frac{a_1}{a_2}\right|\le \frac12.
\]
Since $\lambda_1(\Lambda)\ge\delta$, we have $\|\aa\|_\infty\ge\delta$.
Because $|a_1|\le \tfrac12|a_2|$, it follows that $\|\aa\|_\infty=|a_2|$,
and therefore $
|a_2|\ge\delta.$ 

For the upper bound, by Minkowski's second theorem (Cassels, Chapter III, Section 2), we have
\[
\lambda_1(\Lambda)\lambda_2(\Lambda)\le 1.
\]
Since $\lambda_1(\Lambda)\ge\delta$, this gives $\lambda_2(\Lambda)\le\delta^{-1}$.
The vector $\aa$ from the Gauss-reduced basis satisfies $\|\aa\|_\infty \le
\lambda_2(\Lambda)$ up to a universal constant (see \cite[Section 4.2.3,
Condition (IV)]{HatefiSimmons2024}). Hence there exists a universal
constant $C>0$ such that
\[
\|\aa\|_\infty \le C\delta^{-1}.
\]
Since $\|\aa\|_\infty=|a_2|$, taking $C_0:=C\delta^{-1}$ gives $
|a_2|\le C_0.$ Thus, with $c_0=\delta$ and $C_0=C\delta^{-1}$, we have
\[
c_0\le |a_2|\le C_0,
\qquad
\left|\frac{a_1}{a_2}\right|\le \frac12,
\]
as required.
}

\subsection{Game Balls and Grid States}
\label{subsec:grid-states}

We introduce a correspondence between balls played in the rapid game and grids in $Y$. Namely, each choice of ball $B(x,\rho)$ in the rapid game will correspond to a choice of grid $\Lambda^\gamma := g_t u_x(\Z^2 + (\gamma,0))$, where $t = -\frac12 \log(\rho)$. To be precise, suppose Bob's current ball is
\[
B=B(x_0,\rho), \qquad \rho=e^{-2t}.
\]
The corresponding homogeneous and inhomogeneous states are
\[
\Lambda=g_tu_{x_0}\Z^2, \qquad
\Lambda^\gamma = g_tu_{x_0}(\Z^2+(\gamma,0)).
\]

Suppose Bob announces the parameter $\alpha=e^{-2s}$. If Alice chooses a
centre of the form $y=x_0+e^{-2t}z$, then the corresponding homogeneous
and inhomogeneous states become
\[
g_{t+s}u_y\Z^2 = g_su_z\Lambda,
\qquad
g_{t+s}u_y(\Z^2+(\gamma,0)) = g_su_z\Lambda^\gamma.
\]
Thus, Alice's choice of centre is equivalent to choosing a horocycle
parameter $z$. 
The legality condition is $B(z,\alpha)\subseteq[-1,1]$, which guarantees
that Alice's move remains inside Bob's ball.

\begin{remark}
In the strong game, Bob may choose radii larger than the model values. This
changes the corresponding flow times only by bounded amounts. All
compactness, avoidance, and approximation estimates below are uniform under
such bounded perturbations. This is precisely the role of the strong
formulation in the proof.
\end{remark}







\section{Proof of  Theorem\ref{thm:main}}
\label{sec:proof}
The proof follows the same two-strategy framework as in \cite{HatefiSimmons2024}: Alice alternates between a \emph{default strategy} and an \emph{auxiliary strategy}. The key difference is that the auxiliary strategy is now used only to force a homogeneous $\psi$-approximation, while the return to a bounded region is reserved for the default strategy. This separation is precisely what allows the game to recover the Jarn\'ik--Besicovitch dimension.

\subsection{The Default Strategy}
\label{subsec:default}

Alice's default strategy provides a universal response to every Bob move.  Its primary functions are twofold:
\begin{enumerate}
\item to keep the homogeneous lattice within a bounded, compact region $\mathcal K_\theta$ (or to actively recover it from a deep cusp corridor $\mathcal K_\theta$ back to the universal compact set $\mathcal K_\zeta$); and
\item to maintain the inhomogeneous grid uniformly away from the origin, i.e. inside $\mathcal F_{\zeta'}$.
\end{enumerate}

The latter condition guarantees that every possible outcome of the game lies in $\Bad^\gamma$.

\begin{lemma}
\label{lem:homogeneous}
Fix $0<\alpha<1/4$ and $0<\beta<1$. There exists a strategy for Alice in
the $(\alpha,\beta)$-strong game such that for every $\delta>0$ there is
$\theta=\theta(\delta)>0$ with the following property: if the homogeneous
part starts in $\mathcal K_\delta$, then after finitely many rounds it
enters and remains in $\mathcal K_\zeta$, where $\zeta>0$ is universal
(independent of $\delta$). During this process the homogeneous part stays
in $\mathcal K_\theta$.
\end{lemma}

\begin{proof} Let $\Lambda\in\mathcal K_\delta$. By Lemma~\ref{lem:lattice-minimum}, there exists a primitive vector $\mathbf a=(a_1,a_2)\in\Lambda$ satisfying \[ c_0(\delta)\le |a_2|\le C_0(\delta), \qquad |a_1|\le\frac12|a_2|. \] Alice plays the maximal legal horocycle translation \[ x=\alpha-1 \] followed by the diagonal flow $ t=-\frac12\log\alpha. $  Thus \[ g_tu_{\alpha-1}\mathbf a = \begin{pmatrix} \alpha^{-1/2}(a_1+(1-\alpha)a_2)\\ \alpha^{1/2}a_2 \end{pmatrix}. \] Now let Bob respond by applying $u_y$ followed by $g_s$, where $|y|\le1$ and $0\le s\le-\log\beta$. Writing \[ M=g_su_yg_tu_{\alpha-1}, \] we obtain \[ M\mathbf a = \begin{pmatrix} e^{s}\alpha^{-1/2} \bigl(a_1+(1-\alpha-\alpha y)a_2\bigr)\\ e^{-s}\alpha^{1/2}a_2 \end{pmatrix}. \] Since $|y|\le1$ and $|a_1|\le\frac12|a_2|$, \[ |a_1+(1-\alpha-\alpha y)a_2| \ge \left(\frac12-2\alpha\right)|a_2|. \] As $\alpha<1/4$, the constant $c_\alpha:=\frac12-2\alpha $ is positive. Using also $e^{s}\ge1$ and $|a_2|\ge c_0(\delta)$ gives \[ \|M\mathbf a\|_\infty \ge \alpha^{-1/2}c_\alpha c_0(\delta) =:A_\delta. \] 

On the other hand, every admissible one-step matrix belongs to the compact family \[ \mathscr M = \{\,g_su_yg_tu_{\alpha-1} : |y|\le1,\; 0\le s\le-\log\beta\,\}. \] Hence there exists a constant $L=L(\alpha,\beta) $ such that \[ \|M\| \le L \qquad (M\in\mathscr M). \] Since $ \|\mathbf a\|_\infty \le C_0(\delta),$  we obtain \[ \|M\mathbf a\|_\infty \le LC_0(\delta) =:B_\delta. \] Now let $\mathbf b\in\Lambda$ be any primitive vector independent of $\mathbf a$. Since $ |\det(\mathbf a,\mathbf b)| \geq 1 $ and $\det(M)=1$, \[ 1 \leq |\det(M\mathbf a,M\mathbf b)| \le 2\|M\mathbf a\|_\infty \|M\mathbf b\|_\infty. \] Using the upper bound on $\|M\mathbf a\|_\infty$ yields \[ \|M\mathbf b\|_\infty \ge \frac1{2B_\delta}. \] Therefore every primitive lattice vector of the transformed lattice $M\Lambda$ has norm at least
\[ \min\!\left( A_\delta,\, \frac1{2B_\delta} \right). \] 
Consequently the transformed lattice remains in $\mathcal K_{\theta(\delta)}$, where 
\[ \theta(\delta) = \min\!\left( A_\delta,\, \frac1{2B_\delta} \right). \] 
As in the proof of \cite[Lemma~4.5]{HatefiSimmons2024}, the shortest primitive lattice vector grows until it reaches a universal constant $ \zeta = \frac12\alpha\beta,$ after which it cannot decrease below $\zeta$. Hence after finitely many steps the homogeneous trajectory enters $\mathcal K_\zeta$ and remains there, while throughout the recovery process it stays inside the compact corridor $\mathcal K_{\theta(\delta)}$.
\end{proof}

\begin{lemma}
\label{lem:inhomogeneous}
Assume that the grid state is in $\pi^{-1}(\mathcal K_\delta)\cap\mathcal F_\delta$.
With the same $\alpha,\beta$ as above, there is a strategy for the strong $(\alpha,\beta)$ game taking the grid state to
$\pi^{-1}(\mathcal K_\zeta)\cap\mathcal F_{\zeta'}$ while staying in
$\pi^{-1}(\mathcal K_\theta)\cap\mathcal F_{\theta'}$, where $\theta,\theta'$
depend on $\delta$, and $\zeta,\zeta'$ are universal.
\end{lemma}

The proof is exactly the same as in \cite[Lemma 4.6]{HatefiSimmons2024}, therefore, omitted.

\begin{lemma}\label{lem:default-recovery}
There exists a universal $\zeta>0$ such that for every $\delta>0$ there are
$\theta=\theta(\delta)>0$ and $\theta' = \theta'(\delta) > 0$ with the following property. If the homogeneous
state lies in $\mathcal K_\theta$ and the inhomogeneous state lies in
$\mathcal F_{\theta'}$, then the default strategy returns the homogeneous
state to $\mathcal K_\zeta$ and keeps the inhomogeneous state in
$\mathcal F_{\zeta'}$.
\end{lemma}

\begin{proof}
This is precisely the composition of Lemmas~\ref{lem:homogeneous} and
\ref{lem:inhomogeneous}. Lemma~\ref{lem:homogeneous} gives the return from
$\mathcal K_\theta$ to $\mathcal K_\zeta$ while maintaining the homogeneous
part in a bounded corridor. Lemma~\ref{lem:inhomogeneous} (applied once the
homogeneous part is in $\mathcal K_\zeta$) guarantees that the grid
avoidance is preserved and the grid minimum reaches the universal threshold
$\zeta'$. The corridors $\theta,\theta'$ are chosen large enough to
accommodate the transient cusp excursions.
\end{proof}

\subsection{The Auxiliary Strategy}
\label{subsec:auxiliary}

The auxiliary strategy pushes the homogeneous lattice toward the cusp. It
does not return the lattice to a bounded region; that task is delegated to
the default strategy.

\begin{lemma}\label{lem:auxiliary-cusp}
Assume we are playing the $\Psi$-rapid game, and that the homogeneous game state lies in $\mathcal K_\delta$, and
that the inhomogeneous game state lies in $\mathcal F_\delta$. Suppose Bob
announces $\alpha=e^{-2s}$. Then Alice can choose a horocycle parameter
\[
z=\frac{a_1}{a_2}+\alpha,
\]
where $a=(a_1,a_2)$ is the primitive vector given by
Lemma~\ref{lem:lattice-minimum}, such that the resulting homogeneous
trajectory enters a cusp corridor $\mathcal K_\theta$, with
$\theta=\theta(\delta)$, and the inhomogeneous trajectory remains in
$\mathcal F_{\theta'}$, with $\theta'=\theta'(\delta)$.

Moreover, for every outcome point $x_\infty$ in Alice's ball one has
\[
\normZ{qx_\infty}\lessless q^{-1}\alpha,
\]
for some $q\asymp\rho^{-1/2}$.
\end{lemma}

\begin{proof}
Write Bob's current ball as $B(x_0,\rho)$, and put $\rho=e^{-2T}$. The
associated lattice and grid are
\[
\Lambda=g_Tu_{x_0}\Z^2, \qquad
\Lambda^\gamma=g_Tu_{x_0}(\Z^2+(\gamma,0)).
\]
By Lemma~\ref{lem:lattice-minimum}, choose a primitive vector
$a=(a_1,a_2)\in\Lambda$ with
\[
c_0\le|a_2|\le C_0, \qquad \left|\frac{a_1}{a_2}\right|\le\frac12.
\]
Now let Bob announce $\alpha=e^{-2s}$, and choose
$z=\frac{a_1}{a_2}+\alpha$. If $z'=z+\eta\alpha$ with $|\eta|\le1$, then
\[
a_1-z'a_2=-\alpha a_2-\eta\alpha a_2.
\]
Thus, at the level of the homogeneous flow $0\leq t\leq s$,
\[
g_tu_{z'}a=(e^t(a_1-z'a_2),e^{-t}a_2),
\]
and at the relevant cusp-entry time the vector has size $\asymp e^{-s}$.
Since $\rho=e^{-2T}$, the corresponding denominator satisfies
\[
q\asymp e^T=\rho^{-1/2}.
\]
Therefore every final point $x_\infty$ in Alice's ball satisfies
\[
\normZ{qx_\infty}\lessless q^{-1}\alpha.
\]

For the inhomogeneous grid avoidance condition, note that the wedge products
are preserved:
\[
(g_tu_{z'}a)\wedge(g_tu_{z'}b)=a\wedge b,
\]
where $b\in\Lambda^\gamma$ and $0\leq t\leq s$. 
Let
$\mathscr C=\{(x,y):y\neq0,\ |x|<|y|/2\}$. Choose a primitive
$v\in\Lambda\cap\mathscr C$ and let 
$w\in \Lambda$ so that $\{v, w\}$ forms the unimodular basis. For all sufficiently large integers $N$, the vector
$w+Nv$ also belongs to $\mathscr C$, and
$\{v,w+Nv\}$ is again a unimodular basis. Put $b=w+Nv$. Let $H=\Lambda+r\in \pi^{-1}(\mathcal K_\delta)\cap\mathcal F_\delta$ 
Since $H\in\mathcal F_\delta$, the grid does not contain the origin, so
$r\notin\Lambda$. If both $\Delta_v(H)$ and $\Delta_b(H)$ were zero, then,
because $v$ and $b$ are primitive,
\[
v\wedge r\in\mathbb Z,
\qquad
b\wedge r\in\mathbb Z.
\]
Writing $r=sv+tb$ and using $|v\wedge b|=1$, these two relations imply
$s,t\in\mathbb Z$, hence $r\in\Lambda$, a contradiction. Therefore at least
one of $v,b$ has strictly positive wedge separation from the grid. Choose
that vector and call it $a$. 

Every $h\in H$ has the form
$h=(m+s)v+(n+t)b$, with $m,n\in\Z$ and $|s|,|t| \leq 1/2$. Since $v\wedge b=1$,
\[
\Delta_v(H)=|t|,
\qquad
\Delta_b(H)=|s|.
\]
Since both $v$ and $b$ lie in the cone $|x|<|y|/2$, and their norms are
bounded above by a constant $D_\delta$ depending only on $\delta$. 
Moreover $r\in H$ and $H\in\mathcal F_\delta$, so
\[
\delta\le\|r\|_\infty
\le D_\delta(|s|+|t|)
\le2D_\delta\max\{|s|,|t|\}.
\]
Hence
\[
\max\{\Delta_v(H),\Delta_b(H)\}
\ge \frac{\delta}{2D_\delta}.
\]
Choose $a=v$ or $a=b$ according to which of these two quantities is
larger, and set $\kappa_\delta=\delta/(2D_\delta)$.  Since $a$ lies in the
cone and $\Lambda\in\mathcal K_\delta$, one also has
$|a_2|=\|a\|_\infty\ge\delta$.  

On the other hand, the auxiliary formula
gives $\|g_tu_{z'}a\|_\infty\leq 2C_0$ for $0\leq t\leq s$. Using
\[
|v\wedge w|\leq2\|v\|_\infty\|w\|_\infty,
\]
we obtain
\[
\|g_tu_{z'}b\|_\infty \geq \frac{\mathcal K_\delta}{4C_0}:=\theta'(\delta).
\]
Thus, after replacing $\theta'(\delta)$ by a smaller uniform constant if
necessary, we have $
g_tu_{z'}\Lambda^\gamma\in \FF_{\theta'(\delta)}$ 
for every $0\leq t\leq s$ and every $z'\in B(z,e^{-s})$. This proves that
the inhomogeneous grid remains away from the origin throughout the
auxiliary excursion.

For legality, note that Alice waits for $\alpha\lessless\Psi_{\psi}(\rho)$ rather
than merely for $\alpha\lessless1$. Legality follows from
\[
\left|\frac{a_1}{a_2}\right|\leq\frac12, \qquad z=\frac{a_1}{a_2}+e^{-2s}.
\]
For all sufficiently large $s$, $B(z,e^{-2s})\subseteq[-1,1]$. Hence
Alice's auxiliary ball
\[
B(x_0+e^{-2T}z,e^{-2T}e^{-2s})
\]
is contained in Bob's ball $B(x_0,e^{-2T})$.

For restart, the default strategy is applied after the auxiliary cusp-entry
step. Lemma~\ref{lem:default-recovery} returns the homogeneous component
from the cusp corridor $\mathcal K_\theta$ to the bounded region
$\mathcal K_\zeta$, while maintaining the inhomogeneous avoidance condition.
\end{proof}

\begin{lemma}\label{lem:auxiliary-Psi}
There exist constants $a,b,C>0$ such that the following holds. Let Bob's
current radius be $\rho$, put $Q=\rho^{-1/2}$, and suppose Alice performs
the auxiliary move described above with contraction $\alpha$. Then there
exists an integer $q$ satisfying $aQ\leq q\leq bQ$ such that every possible
outcome point $x_\infty$ inside Alice's auxiliary ball satisfies
\[
\normZ{qx_\infty}\leq Cq^{-1}\alpha.
\]
Consequently, if $\alpha\leq C^{-1}\Phi_\psi(Q)$, then
$\normZ{qx_\infty}<\psi(q)$.
\end{lemma}

\begin{proof}
The first assertion is the calculation in Lemma~\ref{lem:auxiliary-cusp}: the
chosen vector $a=g_Tu_{x_0}(p,q)$ has $q\asymp e^T=Q$, and every final
point satisfies $\normZ{qx_\infty}\leq Cq^{-1}\alpha$. If $aQ\leq q\leq bQ$,
then by definition $q\psi(q)\geq\Phi_\psi(Q)$. Thus
\[
Cq^{-1}\alpha
\leq
Cq^{-1}\cdot C^{-1}\Phi_\psi(Q)
\leq
\psi(q).
\]
This proves the lemma.
\end{proof}

\subsection{Completing the Proof}
\label{subsec:completion}

Fix a non-empty initial interval $I_0$ and a sufficiently small $\beta>0$.
Alice combines the default recovery strategy of Lemma~\ref{lem:default-recovery}
with the auxiliary cusp-entry strategy of Lemma~\ref{lem:auxiliary-cusp} and
the approximation estimate of Lemma~\ref{lem:auxiliary-Psi}.

Let $C$ be the constant in Lemma~\ref{lem:auxiliary-Psi}. Choose the
constant $c_*$ in the definition of $\Psi_{\psi}$ so small that
\[
c_*\leq \frac1{100C}.
\]
Alice follows the default strategy at all ordinary stages. At the $j$th
opportunity when Bob announces a contraction satisfying
\[
\frac{\alpha_n}{\Psi_{\psi}(\rho_n)}\leq \frac1{j^2},
\]
Alice performs the auxiliary move. At such a successful auxiliary stage,
Lemma~\ref{lem:auxiliary-Psi} gives a denominator $q_j\asymp\rho_n^{-1/2}$
such that every possible final outcome satisfies
\[
\normZ{q_jx_\infty}
\leq Cq_j^{-1}\alpha_n.
\]
Using the choice of the auxiliary stage,
\[
\alpha_n\leq \frac1{j^2}c_*\Phi_\psi(\rho_n^{-1/2}),
\]
we get
\[
\normZ{q_jx_\infty}
\leq
\frac{Cc_*}{j^2}q_j^{-1}\Phi_\psi(\rho_n^{-1/2})
\leq
\frac1{100j^2}\psi(q_j).
\]
In particular, $\normZ{q_jx_\infty}<\psi(q_j)$.

After the auxiliary move, Lemma~\ref{lem:default-recovery} returns the
homogeneous part to the bounded region $\mathcal K_\zeta$ and keeps the
inhomogeneous grid inside $F_{\zeta'}$. This is the same restart mechanism
as in \cite{HatefiSimmons2024}, except that the approximation itself was
produced by the one-way cusp-entry move.

There are two cases. If Bob provides only finitely many successful
auxiliary opportunities, then for some $j_0$ and all sufficiently large $n$,
\[
\frac{\alpha_n}{\Psi_{\psi}(\rho_n)}>\frac1{j_0^2}.
\]
Therefore
\[
\liminf_{n\to\infty}\frac{\alpha_n}{\Psi_{\psi}(\rho_n)}>0,
\]
so Bob loses by default in the strong $\Psi_{\psi}$-rapid game. Otherwise
Alice performs infinitely many auxiliary moves. The radii tend to zero in
every non-default play, so the denominators $q_j\asymp\rho_n^{-1/2}$ tend
to infinity. For every $j$,
\[
\normZ{q_jx_\infty}<\psi(q_j).
\]
Hence $x_\infty\in\WW(\psi)$. At the same time, the default strategy and
the inhomogeneous avoidance estimate keep the grid inside $\FF_\varepsilon$
for all sufficiently large times. By Lemma~\ref{lem:dani},
$x_\infty\in\Bad^\gamma$. Thus
\[
x_\infty\in\WW(\psi)\cap\Bad^\gamma.
\]
This proves that $\WW(\psi)\cap\Bad^\gamma$ is strong
$\Psi_{\psi}$-rapid winning.

\section{Proof of Theorem \ref{cor:jb}}\label{sec:dimension}
The upper bound is immediate from
$\Bad^\gamma\cap\WW(\tau)\subseteq\WW(\tau)$ and Theorem~\ref{thm:jb}:
\[
\HD(\Bad^\gamma\cap\WW(\tau)) \leq \HD\WW(\tau) = \frac{2}{\tau+1}.
\]

For the lower bound, suppose $\psi(q)=q^{-\tau}$ with $\tau\geq1$. Then
\[
q\psi(q)=q^{1-\tau}.
\]
With $q\asymp\rho^{-1/2}$,
\[
\Phi_\psi(\rho^{-1/2})
\asymp
(\rho^{-1/2})^{1-\tau}
=
\rho^{(\tau-1)/2}.
\]
Thus $\Psi_{\psi}(\rho)\asymp\rho^\omega$ with $\omega=(\tau-1)/2$. By
Theorem~\ref{thm:main} and Corollary~\ref{cor:gauge-comparable}, we have
\[
\HD(\WW(\tau)\cap\Bad^\gamma)
\geq
\frac1{1+\frac{\tau-1}2}
=
\frac{2}{\tau+1}.
\]
Combining the upper and lower bounds gives
\[
\HD(\WW(\tau)\cap\Bad^\gamma)=\frac{2}{\tau+1}.
\]
This completes the proof.

\section{Concluding remarks}\label{conclusion}

The proof developed in this paper is intrinsically asymmetric: the
homogeneous lattice is used to create the $\psi$-approximation, while the
inhomogeneous grid is kept uniformly away from the origin in order to
guarantee membership in $\Bad^\gamma$. It is therefore natural to ask what
happens when the roles are reversed. Namely, for a fixed
$\gamma_0\in\R/\Z$, one may consider
\[
W^{\gamma_0}(\psi)
:=
\left\{
x\in\R:
\|qx-\gamma_0\|_{\Z}<\psi(q)
\ \text{for infinitely many }q \in\N
\right\}
\]
and ask whether the sets
\[
W^{\gamma_0}(\psi)\cap\Bad
\qquad\text{or, more generally,}\qquad
W^{\gamma_0}(\psi)\cap\bigcap_{j=1}^r\Bad^{\gamma_j}
\]
admit an analogous rapid-winning property. There is an immediate arithmetic obstruction which shows that such a
statement cannot hold without assumptions on the shifts. If, for instance, 
\[
\gamma_j\equiv k\gamma_0\pmod 1
\]
for some integer $k\ge1$ and $q\psi(q)\to0$, then
\[
W^{\gamma_0}(\psi)\cap\Bad^{\gamma_j}=\varnothing.
\]
Indeed, from
$\|qx-\gamma_0\|_{\Z}<\psi(q)$ one obtains
\[
kq\,\|kqx-\gamma_j\|_{\Z}
\le k^2q\psi(q)\to0.
\]
In particular, if $\gamma_0$ is rational and $\psi(q)=q^{-\tau}$ with
$\tau>1$, then
\[
W^{\gamma_0}(\tau)\cap\Bad=\varnothing.
\]
Thus the reversed problem is genuinely sensitive to arithmetic relations
between the shifts.

For suitable shifts, however, it seems plausible that the
$\Psi$-rapid-game approach can be adapted. The main missing ingredient is
a new \emph{affine cusp-entry lemma}: one would need to make a point of the
$\gamma_0$-grid short, thereby producing an inhomogeneous approximation,
while keeping the homogeneous lattice and all the forbidden grids
corresponding to $\gamma_1,\ldots,\gamma_r$ uniformly away from the
origin. The wedge-product argument used in the present paper does not
directly provide this, since the relevant determinants involve quantities
of the form
\[
\|s\gamma_0-q\gamma_j\|_{\Z},
\]
whose uniform lower bounds depend on the arithmetic of the tuple of
shifts.

This leads to several natural problems for future work:
\begin{enumerate}
\item determine sharp arithmetic conditions on
$(\gamma_0,\gamma_1,\ldots,\gamma_r)$ under which
\[
\dim_H\left(
W^{\gamma_0}(\tau)\cap\bigcap_{j=1}^r\Bad^{\gamma_j}
\right)
=
\frac{2}{\tau+1};
\]
\item develop a multi-grid version of the auxiliary and recovery
strategies capable of handling finitely, or possibly countably, many
inhomogeneous bad sets simultaneously;
\item investigate whether analogous rapid-winning statements hold in
higher-dimensional, weighted, or systems-of-linear-forms settings.
\end{enumerate}
These questions suggest that the $\Psi$-rapid game may apply more broadly to Diophantine problems combining approximation and avoidance at different scales. Recently, Ward~\cite{Ward2026} introduced a higher-dimensional rapid absolute game, inspired by Hatefi--Simmons~\cite{HatefiSimmons2024}, and used it with the Simplex lemma to obtain positive, and asymptotically full, Hausdorff-dimension results for level sets of the $d$-dimensional Lagrange constant. This provides a first higher-dimensional precedent for rapid-type methods. It remains to develop a genuinely $\Psi$-rapid theory for prescribed approximation functions, weighted approximation, and systems of linear~forms.

\end{document}